\documentclass[11pt,reqno]{amsart}
\usepackage[dvipsnames]{xcolor}
\usepackage{setspace}
\usepackage[hang,flushmargin,symbol*]{footmisc}
\usepackage{amsmath}
\usepackage{amsthm}
\usepackage{amssymb}
\usepackage{mathtools}
\usepackage{enumitem}
\usepackage{calc}
\usepackage{graphicx}
\usepackage{caption}
\usepackage[labelformat=simple,labelfont={}]{subcaption}
\usepackage{tikz}
\usepackage{tikz-cd}
\usetikzlibrary{decorations.markings}
\usetikzlibrary{arrows,shapes,positioning}
\usepackage{url}
\usepackage{array}
\usepackage{graphicx}
\usepackage{color}
\usepackage{mathrsfs}

\usepackage{verbatim}

\usepackage{subfiles}

\usepackage[disable]{todonotes}

\usepackage{dashbox}

\theoremstyle{definition}
\newtheorem{theorem}{Theorem}[section]
\newtheorem{lemma}[theorem]{Lemma}

\newtheorem{corollary}[theorem]{Corollary}
\newtheorem{proposition}[theorem]{Proposition}

\newtheorem{definition}[theorem]{Definition}
\newtheorem{example}[theorem]{Example}
\newtheorem{remark}[theorem]{Remark}
\newtheorem{question}[theorem]{Question}

\usepackage{marginnote}

\usepackage{amsfonts}
\usepackage{amsmath}
\usepackage{amssymb}
\usepackage{mathrsfs}
\usepackage{tikz}
\usepackage{tikz-cd}
\usepackage[colorlinks=true]{hyperref}
\usepackage{ mathdots }
\usepackage{dutchcal}
\usepackage{verbatim}
\usepackage{amsmath,amsthm}
\usepackage{extarrows}

\usepackage{xparse}

\usepackage{multirow}

\usepackage[normalem]{ulem}

\newcommand{\ZZ}{\mathbf{Z}}

\newcommand{\NN}{{\mathbf{N}}}

\newcommand{\Hom}{\textrm{Hom}}
\newcommand{\HHom}{\underline{\textrm{Hom}}}

\newcommand{\IIsom}{\underline{\textrm{Isom}}}
\newcommand{\EExt}{\underline{\textrm{Ext}}}

\newcommand{\Der}{\textrm{Der}}

\newcommand{\dd}[1]{\dh_{#1}}
\newcommand{\EExal}{\underline{\textrm{Exal}}}
\newcommand{\DDef}{\underline{\textrm{Def}}}
\newcommand{\Exal}{\textrm{Exal}}
\newcommand{\Ext}{\textrm{Ext}}

\newcommand{\hh}[1]{\mathcal{h}_{#1}}

\newcommand{\pb}{{\arrow[dr, phantom, very near start, "\ulcorner"]}}

\newcommand{\longequals}{\xlongequal{\: \:}}

\newcommand{\Split}{\underline{\textrm{Split}}}

\newcommand{\ccx}[1]{\mathbf{L}_{#1}}

\renewcommand{\tilde}[1]{\ensuremath{\widetilde{#1}}}

\newcommand{\cal}[1]{\ensuremath{\mathcal{#1}}}

\newcommand{\onlyinsubfile}[1]{#1}
\newcommand{\notinsubfile}[1]{}

\newcommand{\longsimeq}{\overset{\sim}{\longrightarrow}}

\usepackage{stmaryrd}

\definecolor{sebgreen1}{rgb}{0.019,0.317,0.149}
\definecolor{sebgreen2}{rgb}{0.784,0.952,0.780}

\newcommand{\scr}[1]{{\ensuremath{\mathscr{#1}}}}

\newcommand{\Aalg}{A-{\rm alg}}
\newcommand{\Amod}{A-{\rm mod}}
\newcommand{\AAalg}{A'-{\rm alg}}

\makeatletter
\let\@wraptoccontribs\wraptoccontribs
\makeatother

\renewcommand{\bar}[1]{\overline{#1}}

\title[Deformations of Algebras]{Deformations of Algebras with 2-Extensions}
\author[Leo Herr]{Leo Herr \\ Appendix joint with Jonathan Wise}
\date{\today}

\address{University of Utah}
\email{herr@math.utah.edu}

\address{University of Colorado Boulder}

\begin{document}

\maketitle

\renewcommand{\onlyinsubfile}[1]{}
\renewcommand{\notinsubfile}[1]{#1}

\setcounter{section}{-1}

\begin{abstract}
We apply the notion of 2-extensions of algebras to the deformation theory of algebras. After standard results on butterflies between 2-extensions, we use this (2, 0)-category to give three perspectives on the deformation theory of algebras. We conclude by fixing an error in the literature. 
\end{abstract}

\section{Introduction}

This paper
\begin{itemize}
    \item sets 2-extensions of algebras on the same footing as that of modules \cite{mymoduledefms}, \cite{wisebutterflies},
    \item obtains the theorems of \cite{mymoduledefms} about modules for algebras,
    \item fixes a subtle flaw (Example \ref{freeproductnoiso}) in \cite{wisealgs1}.
\end{itemize}

Fix a squarezero extension of sheaves of commutative rings
\[0 \to I \to A' \to A \to 0,\]
an $A$-algebra $B$, a $B$-module $M$, and a map $\varphi : I \to M$, all in a topos $E$. 
A \textit{deformation} of $B$ is a commutative diagram 
\[\begin{tikzcd}
0 \ar[r] &I \ar[r] \ar[d, "\varphi", swap] &A' \ar[r] \ar[d, dashed]  &A \ar[r] \ar[d]  &0 \\
0 \ar[r] &M \ar[r, dashed] &B' \ar[r, dashed] &B \ar[r] &0.
\end{tikzcd}\]

\begin{question}[{\cite{gerstenhaber1}, \cite[\S 2.1.1]{illusie1}}]\label{q:howmanydefms?}
Do deformations $B'$ exist? If so, how many deformations are there? What is the structure of the set of all deformations $B'$? 
\end{question}

Write $\DDef(-, \varphi, M)$ for the fibered category over $\Aalg/B$ of deformations $B'$. J. Wise introduced a topology on the category $\Aalg/B$ of $A$-algebras over $B$ in which $\DDef(-, \varphi, M)$ forms a gerbe \cite[\S 4.1]{wisealgs2}. 

This gerbe is banded by the sheaf of $A$-derivations
\[\dd{M}(C) := \Der_A(C, M),\]
yielding a class
\[\DDef(-, \varphi, M) \in H^2(\Aalg/B, \dd{M}).\]
This class vanishes if and only if there is a solution $B'$ to Question \ref{q:howmanydefms?}. 

Wise previously computed that the cohomology of $\dd{M}$ on the site $\Aalg$ gives algebra extensions
\[\Exal_A(B, M) \simeq H^1(\Aalg/B, \dd{M})\]
of $B$ by $M$. An $A$-algebra extension $\xi$ of $B$ by $M$ corresponds to the $\dd{M}$-torsor of its splittings. 

\begin{question}
Is there an analogous notion of 2-extensions of algebras $\Exal^2_A(B, M)$ such that
\[\Exal^2_A(B, M) \simeq H^2(\Aalg/B, \dd{M})?\]
If so, which 2-extension corresponds to $\DDef(-, \varphi, M)$?
\end{question}

2-extensions of algebras were defined in \cite{gerstenhaber2} and \cite{lichtenbaumschlessinger}, endowing 2-extensions of abelian groups
\[\xi : \qquad 0 \to M \to N \to R \to B \to 0\]
with appropriate multiplicative structure. They took morphisms to be multiplicative chain maps. We want splittings of $\xi$ to form a gerbe under $\dd{M}$, but as such they do not. 

We define a $(2, 0)$-category $\EExal_A^2(B, M)$ to have 2-extensions with $A$-algebra structure \S \ref{ss:algstr} as objects, but \emph{butterflies} between them as 1-morphisms \cite{butterfliesoriginalnoohialdrovandi}. Splittings of $\xi$ now form a $\dd{M}$-gerbe by Proposition \ref{prop:isomdMgerbe}. Writing $\Exal^2_A(B, M)$ for the group of isomorphism classes, we obtain the expected isomorphism
\[\Exal_A^2(B, M) \simeq H^2(\Aalg/B, \dd{M})\]
by adapting work of Wise:

\begin{proposition}[{\cite[Theorem 2]{wisealgs1}}, {\cite[Lemma 4.2]{wisealgs2}}]\label{prop:explicitcohomisexal}
There are isomorphisms 
\[\Exal_A(B, M) \simeq H^1(\Aalg/B, \dd{M})\]
\[\Exal_A^2(B, M) \simeq H^2(\Aalg/B, \dd{M})\]
sending algebra extensions $\gamma$ and 2-extensions $\xi$ to their $\dd{M}$-torsor/gerbe of splittings. 
\end{proposition}

This allows us to rewrite the transitivity sequence of Illusie \cite[\S II.2.1]{illusie1} in terms of specific 2-extensions of algebras. We identify the 2-extension corresponding to $\DDef(-, \varphi, M)$ in \S \ref{ss:algdefms}.

Appendix \ref{a:fixflaw} fixes a flaw in the proof of the main theorem of \cite{wisealgs1}:

\begin{theorem}[=Theorem \ref{thm:algcohomisexal}]\label{introthm:algcohomisexal}
Cohomology of the sheaf $\dd{M}$ on $\Aalg/B$ corresponds to the Lichtenbaum-Schlessinger functors $T^0, T^1, T^2, \dots$;
i.e., to $\Ext$ with the cotangent complex: 
\[
\begin{split}
H^p(\Aalg/B, \dd{M}) &= \Ext^p_B(\ccx{B/A}, M)  \\
    &=T^p(B/A, M).
\end{split}
\]
\end{theorem}

Theorem \ref{thm:injdervanishing} lets one prove this theorem the same way. The appendix is independent, so we may use these identifications throughout the paper.

\subsection*{Conventions}

Fix a topos $E$. All algebraic structures are internal to $E$ -- even $\ZZ$ denotes the constant initial ring. All ``rings'' and ``algebras'' are commutative and unital unless otherwise specified. Let $A$ be a ring and $B$ an $A$-algebra.

\subsection*{Literature Summary}

Butterflies were introduced in their present form by the seminal \cite{butterfliesoriginalnoohialdrovandi} and its sequels. They appear in a different form in \cite{SGA7-I}. Splittings of a 2-extension appear a couple years earlier in \cite{gerstenhaber2} as ``solutions.'' Since any invertible butterfly $\xi \simeq \eta$ can be written as a splitting of $\xi-\eta$, one could even attribute their origin to Gerstenhaber. The author is not aware of an earlier source.

The 2-extensions we use appear as such in \cite{lichtenbaumschlessinger} and as ``admissible sequences'' in \cite{gerstenhaber2}. 

The middle two terms of a 2-extension are here called ``crossed rings'' $f : N \to R$. They were known as ``crossed modules'' in \cite{aldrovandianncats} and ``quasi-ideals'' in \cite{ringgroupoiddrinfeld}. Op. cit. shows that crossed rings are equivalent to DG rings concentrated in degrees $[-1, 0]$. Lemma 3.3.1 shows they are also equivalent to simplicial rings $\partial_0, \partial_1 : A_1 \rightrightarrows A_0$ concentrated in two degrees such that $\ker \partial_0 \cdot \ker \partial_1 = 0$ (taking $A_1 = R + N, A_0 = R$). These are all equivalent to the 1-category of ring groupoids. The butterflies defined here are described in \S 4. 

Our article does not overlap op. cit. much more than it does \cite{aldrovandianncats}, \cite{lichtenbaumschlessinger}, or \cite{gerstenhaber2}. Nevertheless, Drinfeld offers valuable context in a series of categorical equivalences. 

Though the above papers work primarily in the topos $E = (Sets)$, it is understood that minor modifications remove this restriction. 

A similar topology was introduced in \cite{wisealgs1}; an identical one was used in \cite{mymoduledefms}. 

This paper applies insights of Wise, Illusie, Gerstenhaber, Lichtenbaum, Schlessinger, Aldrovandi, Noohi, Grothendieck, Drinfeld, etc. to the deformation theory of algebras and 2-extensions in an elementary fashion.

\subsection*{Acknowledgments}

This article arose from an office discussion between the author, Benjamin Briggs, Josh Pollitz, and Daniel McCormick. All gracefully declined to be coauthors. The author was trying to guess the correct notion of 2-extension when B. Briggs told him it could be found in \cite{lichtenbaumschlessinger} and even \cite{gerstenhaber2}. This notion obtained the results of \cite{mymoduledefms} in the same fashion and restated some of those of \cite{wisealgs2}.

The authors of Appendix \ref{a:fixflaw} discovered the mistake in \cite{wisealgs1} while the first author was a graduate student of the second at the University of Colorado Boulder. Simplicial algebra can no longer be avoided, but sublimated into \v Cech cohomology. 

J. Wise pioneered this approach to deformation theory. The author is grateful to him in general, but also specifically for useful conversations and finding mistakes in drafts of the present article. 

The author is also grateful to the NSF for partial financial support through RTG Grant \#1840190.

\section{2-extensions}

This section defines a $(2, 0)$-category $\EExal^2_A(B, M)$ of 2-extensions of $A$-algebras with butterflies as 1-morphisms. 

A short exact sequence
\[\zeta : 0 \to M \to B' \to B \to 0\]
with $B' \to B$ $A$-algebras is called an \textit{extension of} $A$-{algebras} if $M^2 = 0$ in the multiplication of $B'$. Morphisms between such extensions are chain maps, defining a category $\EExal_A(B, M)$ of extensions of $B$ by $M$. 

An automorphism of $\zeta$ gives a derivation $B \to M$ \cite[4.3.4]{lichtenbaumschlessinger}. In a topology to be specified on $A$-algebras over $B$, isomorphisms of $\zeta$ with the trivial extension $B + \epsilon M$ form a torsor under derivations. We can identify extensions with sheaf cohomology 
\[H^1(\Aalg/B, \dd{M}) \simeq \EExal_A(B, M).\]

To explore $\dd{M}$-gerbes and move further down the long exact sequence of sheaf cohomology, we need a 2-category enriched over extensions of algebras. As for extensions of modules \cite{mymoduledefms}, \cite{wisebutterflies}, objects will be longer sequences
\[0 \to M \to N \to R \to B \to 0.\]
We need to generalize the condition $M^2=0$. 

\begin{definition}
Let $R$ be a ring and $N$ an $R$-module and non-unital commutative algebra. A \emph{crossed ring} $f : N \to R$ is a homomorphism of non-unital commutative algebras compatible with the $R$-module structure; for all $x, y \in N$,
\begin{equation}\label{eqn:crossedmod}
xy = f(x).y \qquad \in N.
\end{equation}
\end{definition}

Gerstenhaber said these maps $N \to R$ \emph{conform} \cite[Definition 4]{gerstenhaber2}. Any ideal $K \subseteq R$ is a crossed ring. If $R \to S$ is a map of rings, $N \otimes_R S \to S$ is also a crossed ring. This gives a left adjoint in the usual way. Beware that there need not be a map $S \dashrightarrow N \otimes_R S$ and this tensor product is not the pushout in the category of non-unital commutative algebras.

\begin{remark}

Crossed rings $f : N \to R$ are equivalent to $R$-module maps such that $f(x).y = f(y).x$. Given such a map, one can define a non-unital commutative algebra structure on $N$. Chain maps between such objects are multiplicative. Crossed rings were defined this way in \cite[Definition 2.1.1]{lichtenbaumschlessinger}, \cite[Definition 2.1.1]{aldrovandianncats}, and \cite[3.3.1]{ringgroupoiddrinfeld}. 

\end{remark}

\begin{definition}

After Gerstenhaber \cite[pg. 4]{gerstenhaber2}, a 2-extension $\xi$ is an exact sequence of groups
\begin{equation}\label{eqn:2extn}
\begin{tikzcd}
\xi:     &0 \ar[r]   &M \ar[r, "e"]      &N \ar[r, "f"]      &R \ar[r, "g"]      &B \arrow[r]      &0,
\end{tikzcd}
\end{equation}
where $g : R \to B$ is a surjection of rings, $f : N \to R$ is a crossed ring, and $e$ an $R$-module homomorphism. 

\end{definition}

Necessarily $M$ is a $B$-module and the product it inherits from $N$ is zero. The ``trivial'' 2-extension is given by
\begin{equation}\label{eqn:trivial2extn}
\begin{tikzcd}
\overline{0}: &0 \ar[r] &M \ar[r, equals] &M \ar[r, "0"] &B \ar[r, equals] &B \ar[r] &0.
\end{tikzcd}    
\end{equation}
We frequently use this notation \eqref{eqn:2extn} for a 2-extension without further mention, adding $'$s or $\tilde{\cdot}$s as necessary.

\begin{definition}

A \emph{butterfly} \cite{nlab:butterfly} between 2-extensions is a diagram
\begin{equation}\label{eqn:butterfly}
\begin{tikzcd}
\xi : &0 \ar[r]   &M \ar[r] \ar[dd]   &N\ar[dr] \ar[rr]  &&R \ar[r]     &B \ar[dd] \ar[r]  &0  \\
&    & & &Q \ar[ur] \ar[dr] \ar[d, phantom, "-"]  & & \\
\eta : &0 \ar[r]   &M'  \ar[r] &N'\ar[ur]  \ar[rr] &\phantom{a} &R' \ar[r]    &B' \ar[r] &0
\end{tikzcd}    
\begin{tikzpicture}[scale=.75]
\draw[->] (0, 0) to (0, .7);
\draw[->] (0, 0) to (.7, 0);
\draw[->] (0, 0) to (0, -.7);
\draw[->] (0, 0) to (-.7, 0);
\draw (0, 0) circle (.8em);
\node at (0, 1){N};
\node at (1, 0){E};
\node at (-1, 0){W};
\node at (0, -1){S};
\end{tikzpicture}
\end{equation}
such that 
\begin{enumerate}
    \item The $N'QR'$ triangle marked with ``$-$'' \emph{anticommutes}. The rest of the diagram commutes.
    \item\label{eitem:SWNEexact} the ``SW-NE'' diagonal sequence 
    \[0 \to N' \to Q \to R \to 0\]
    is exact. 
    \item $Q \to R, Q \to R'$ are maps of rings. 
    \item\label{eitem:crossmodcptibility} The maps $N \to Q$, $N' \to Q$ are crossed rings which are induced from their $R, R'$ crossed ring structures. 
    \item The composite $N \to Q \to R'$ is zero.
\end{enumerate}
\end{definition}

Such diagrams are abusively denoted by $Q : \xi \to \eta$. The anticommutativity convention comes from \cite[Equation (1)]{wisebutterflies}, and ensures composition of butterflies is naturally defined.

The composite of two analogously named butterflies $Q : \xi \to \xi', Q' : \xi' \to \xi''$ is
\[\begin{tikzcd}
0 \ar[r]  &M \ar[r] \ar[dd]  &N \ar[rr] \ar[dr]  &&R \ar[r]  &B \ar[r] \ar[dd]  &0 \\
 & & &(Q \times_{R'} Q')/N' \ar[ur] \ar[dr]\\
0  \ar[r] &M'' \ar[r]  &N'' \ar[rr] \ar[ur]  &\phantom{a} \ar[from=u, phantom, "-"]  &R'' \ar[r] &B'' \ar[r] &0.
\end{tikzcd}\]
with natural maps as in \cite[\S 3]{wisebutterflies}. A morphism \eqref{eqn:butterfly} is invertible if and only if it restricts to isomorphisms $M \longsimeq M', B \longsimeq B'$ \cite[Proposition 3]{wisebutterflies}. In that case, the inverse is the same butterfly upside-down up to signs. 

Maps $Q \to Q'$ between butterflies must commute with all the structure maps and are necessarily isomorphisms. Our labeling by NSEW conflicts with those writing butterflies left to right. 

Write $\EExal^2(B, M)$ for the 2-category with objects 2-extensions and morphisms butterflies that restrict to isomorphisms $M \longsimeq M'$, $B \longsimeq B'$. All such butterflies are invertible. The NW-SE sequence
\[0 \to N \to Q \to R' \to 0\]
is automatically exact by commutativity of Diagram \eqref{eqn:butterfly} and exactness of the SW-NE sequence \eqref{eitem:SWNEexact} \cite[Proposition 3]{wisebutterflies}. 

The category $\EExal^2(B, M)$ is therefore a (2, 0)-category. The diagonals need not be squarezero extensions of algebras.

Thinking of every object in \eqref{eqn:butterfly} as an abelian group, one obtains a forgetful functor $\EExal^2(B, M) \to \EExt^2(B, M)$.

If $N$ is a non-unital commutative algebra and $R$-module, write $R + N$ for the group $R \oplus N$ with multiplication
\[(r, n)(r', n') := (rr', rn' + r'n + nn').\]
A chain map between two 2-extensions $\xi$, $\eta$ respecting multiplication gives a butterfly \[\begin{tikzcd}
\xi :   &0 \ar[r]      &M \ar[r] \ar[dd]  &N  \ar[dr, "(+-)", swap] \ar[rr]  &   &R \ar[r]  &B \ar[dd] \ar[r]  &0      \\
        &       &   &   &R + N' \ar[ur] \ar[dr]    \\
\eta :  &0 \ar[r]      &M' \ar[r] &N' \ar[ur, "-"] \ar[rr] &\phantom{a} \ar[from=u, phantom, "-"]   &R' \ar[r] &B' \ar[r] &0.  
\end{tikzcd}\]
The maps to the second factor of $R + N'$ are both the negative $-$ of the natural maps. A butterfly arises from such a chain map if and only if it is split $Q \simeq R + N'$.

\begin{example}

Suppose given a chain map and a butterfly:
\[\begin{tikzcd}
0 \ar[r] &M \ar[r] \ar[d] &N \ar[rr] \ar[d] & &R \ar[r] \ar[d] &B \ar[r] \ar[d] &0 \\
0\ar[r] &M'\ar[r] \ar[dd] &N'\ar[rr] \ar[dr] & &R'\ar[r] &B'\ar[r] \ar[dd] &0 \\
 & & &Q \ar[ur] \ar[dr] \\
0\ar[r] &M''\ar[r] &N''\ar[rr] \ar[ur] &\phantom{a} \ar[from=u, phantom, "-"]&R''\ar[r] &B''\ar[r] &0.
\end{tikzcd}\]
Their composite is given by the pullback of the SW-NE sequence of $Q$ along $R \to R'$:
\[\begin{tikzcd}
0\ar[r] &M \ar[r] \ar[dd] &N \ar[rr] \ar[dr] & &R \ar[r] &B \ar[r] \ar[dd] &0 \\
 & & &Q \times_{R'} R \ar[ur] \ar[dr] \\
0\ar[r] &M''\ar[r] &N''\ar[rr] \ar[ur] &\phantom{a} \ar[from=u, phantom, "-"]&R''\ar[r] &B''\ar[r] &0.
\end{tikzcd}\]

\end{example}

One can pull back 2-extensions along maps $B' \to B$:
\[\begin{tikzcd}
0 \ar[r]   &M \ar[r] \ar[d, equals]  &N \ar[r]\ar[d, equals]  &R \times_B B' \ar[r]\ar[d] \pb  &B' \ar[r]\ar[r]\ar[d] &0 \\
0 \ar[r]   &M \ar[r]  &N  \ar[r] &R  \ar[r] &B \ar[r]  &0  
\end{tikzcd}\]
and push out 2-extensions along $M \to M'$ 
\[\begin{tikzcd}
0\ar[r]   &M\ar[r]\ar[d] \ar[dr, phantom, very near end, "\lrcorner"]    &N \ar[r]\ar[d] &R\ar[r]\ar[d, equals]  &B\ar[r]\ar[d, equals]  &0  \\
0\ar[r]   &M'\ar[r]   &N \oplus_M M'\ar[r]    &R\ar[r]  &B\ar[r]  &0.
\end{tikzcd}\]
The butterflies associated to these chain maps have the expected universal properties. 

The 2-categories $\EExal^2(B, M)$ vary contravariantly in $B$, forming a fibered 2-category $\EExal^2(-, M)$ over $\Aalg/B$. Pushforward and pullback also equip $\EExal^2(B, M)$ with a Picard category/strict 2-group structure via the diagonal $\Delta : B \to B \times B$ and addition maps $M \oplus M \to M$. The maps varying $B, M$ respect addition. Write $\Exal^2(B, M)$ for the group of isomorphism classes.

Write a pair of butterflies $Q : \xi \to \eta$, $\tilde Q : \tilde \xi \to \tilde \eta$ as in \eqref{eqn:butterfly}, the latter with $\tilde \cdot$'s on everything. Obtain a butterfly $Q \times \tilde Q : \xi \times \tilde \xi \to \eta \times \tilde \eta$:
\begin{equation}\label{eqn:productbutterfly}
\begin{tikzcd}[column sep=small]
0 \ar[r] &M \oplus \tilde M \ar[r] \ar[dd] &N \oplus \tilde N \ar[rr]\ar[dr] & &R \times \tilde R \ar[r] &B \times \tilde B \ar[r]\ar[dd] &0\\
 & & &Q \times \tilde Q \ar[ur] \ar[dr] \\
0 \ar[r] &M' \oplus \tilde M' \ar[r] &N' \oplus \tilde N' \ar[rr] \ar[ur] &\phantom{a} \ar[from=u, phantom, "-"] &R' \times \tilde R' \ar[r] &B' \times \tilde B'\ar[r] &0.
\end{tikzcd}
\end{equation}

If $Q, \tilde Q$ restrict to identities on the outer terms and so lie in $\EExal^2(B, M)$, define $Q + \tilde Q : \xi + \tilde \xi \to \eta + \tilde \eta$ by pulling back the product $Q \times \tilde Q$ \eqref{eqn:productbutterfly} along $\Delta : B \to B \times B$ and pushing out along $+ : M \oplus M \to M$. The middle term $Q + \tilde Q$ is the homology of the sequence
\[M \overset{(+ -)}{\longrightarrow} Q \times \tilde Q \overset{(+ -)}{\longrightarrow} B.\]
This makes the sum $+$ of two 2-extensions a bifunctor. Write $Q + \tilde \xi$ for the sum $Q + id_{\tilde \xi}$ with the identity.

\begin{remark}

If $\xi$ is a 2-extension, there is a canonical splitting of $\xi-\xi$:
\[\begin{tikzcd}
0  \ar[r] &M \ar[r] \ar[dd, equals] &N^{\oplus 2}/\Delta M  \ar[rr] \ar[dr] & &R + P \ar[r] &B \ar[r] \ar[dd, equals] &0 \\
 & & &R + N \ar[dr] \ar[ur] \\
0 \ar[r] &M \ar[r, equals] &M \ar[rr, "0", swap] \ar[ur] & &B \ar[r, equals] &B \ar[r] &0.
\end{tikzcd}\]
We identified $R \times_B R \simeq R + P$ via the shearing map $(r, r') \mapsto (r, r-r')$. The map $M \to N^{\oplus 2}/\Delta M$ is the inclusion into the first factor, while $N^{\oplus 2}/\Delta M \to R + N$ sends $(n, n') \mapsto (n, n-n')$.

\end{remark}

\begin{remark}\label{rmk:buttaut=exal}

Automorphisms of a butterfly $\xi$ are canonically identified with $\EExal(B, M)$. One can simply subtract $\xi$ from $Q : \xi \simeq \xi$ to get an automorphism of $\xi - \xi$. Use the canonical identifications $\xi \simeq \bar 0$ to get an automorphism of $\bar 0$, canonically an extension in $\EExal(B, M)$. 

\end{remark}

Two 2-extensions $\xi, \xi' \in \Exal^2(B, M)$ are called ``equivalent'' in \cite{lichtenbaumschlessinger}, \cite[pg. 4]{gerstenhaber2} if there is a finite path between them of chain maps between 2-extensions:
\[\xi \leftarrow \xi_1 \to \xi_2 \leftarrow \cdots \to \xi',\]
each restricting to the identity on $B$ and $M$. By \cite[Proposition 3]{wisebutterflies}, these chain maps induce invertible butterflies. The group $Ex^2(B, M)$ of equivalence classes of 2-extensions \cite[Definition 4.1.1]{lichtenbaumschlessinger}, \cite[pg. 4]{gerstenhaber2} coincides with our notion up to isomorphism:

\begin{lemma}\label{lem:LSequiv=isom}
Equivalence classes of 2-extensions $Ex^2(B, M)$ coincide with isomorphism classes $\Exal^2(B, M)$ of 2-extensions. 
\end{lemma}

\begin{proof}

Chain maps induce butterflies, so paths of chain maps between $\xi, \xi'$ induce paths of butterflies. 

Given a butterfly $Q : \xi \simeq \xi'$ as in \eqref{eqn:butterfly} with $M = M', B = B'$, we must produce a path of chain maps between $\xi$ and $\xi'$. We construct a ``free'' 2-extension \cite[Definition 2.1.3]{lichtenbaumschlessinger} mapping to both. 

Choose a polynomial ring with a surjection $\ZZ[S] \to B$ with a lift:
\begin{equation}\label{eqn:freesplitsurj}
\begin{tikzcd}
        &Q \ar[d]      \\
\ZZ[S] \ar[r] \ar[ur, dashed]        &B.
\end{tikzcd}    
\end{equation}
For example, take $S = Q$. One cannot take an arbitrary polynomial ring surjecting onto $B$ as in op. cit. because not all surjections of sheaves of sets have sections. Write 
\[K = \ker (\ZZ[S] \to B), L = \ker(Q \to B),  Z = \ker(R \to B), Z' = \ker(R' \to B).\]
The lift \eqref{eqn:freesplitsurj} gives a map $K \to L$. The ideals $N, N'$ lie inside $L$. There is a commutative diagram
\[\begin{tikzcd}
0 \ar[r] &M \ar[r] \ar[d, equals] &N \ar[r] &Z \ar[r] &0 \\
0 \ar[r] &M \ar[r] \ar[d, equals] &N \oplus N' \ar[d] \ar[u] \ar[r] &L \ar[d] \ar[u]  \ar[r] &0 \\
0 \ar[r] &M \ar[r] &N' \ar[r] &Z' \ar[r] &0.
\end{tikzcd}\]
Choose a free $\ZZ[S]$-module $F$ surjecting onto $K$ with a lift
\[\begin{tikzcd}
               & &N \oplus N' \ar[d]         \\
F \ar[r] \ar[urr, dashed, bend left=15]       &K \ar[r]      &L.
\end{tikzcd}\]    
Write $U = \ker(F \to K)$ and $\bar F = F \oplus_U M$ for the pushout:
\[\begin{tikzcd}
0\ar[r] &M \ar[r] \ar[d, equals] &\bar F  \ar[r] \ar[d] \pb &K \ar[r] \ar[d] &0 \\
0\ar[r] &M\ar[r] &N \oplus N'\ar[r] &L\ar[r] &0.
\end{tikzcd}\]
Give $\bar F = K \times_L (N \oplus N')$ the natural non-unital commutative algebra structure. Taking the cup product of each with its extension of algebras, we get chain maps
\[\begin{tikzcd}
0 \ar[r] &M \ar[r] &N \ar[r] &R \ar[r] &B \ar[r] &0 \\
0 \ar[r] &M \ar[r] \ar[d, equals] \ar[u, equals] &\bar F \ar[r] \ar[d] \ar[u] &\ZZ[S]  \ar[d] \ar[u]\ar[r] &B  \ar[d, equals] \ar[u, equals] \ar[r] &0 \\
0 \ar[r] &M \ar[r] &N' \ar[r] &R' \ar[r] &B \ar[r] &0.
\end{tikzcd}\]
The result is therefore a 2-extension mapping to both $\xi$ and $\xi'$. 

\end{proof}

\subsection{$A$-algebra structure on 2-extensions}\label{ss:algstr}

A squarezero extension 
\[\zeta : \qquad 0 \to M \to B' \to B \to 0\]
of an $A$-algebra $B$ is an $A$-algebra extension if the map $A \to B$ is given a factorization $A \dashrightarrow B'$ through the extension. This equivalently trivializes the pullback of $\zeta$ to $A$. 

An $A$-algebra structure on a 2-extension 
\[\xi : \qquad 0 \to M \to N \to R \to B \to 0\] 
will be a trivialization $Q : \xi \simeq \bar 0$:
\[\begin{tikzcd}
0 \ar[r] &M \ar[r] \ar[dd, equals] &N \ar[rr] \ar[dr] &&R \ar[r] &B \ar[r] \ar[dd, equals] &0        \\
 &&&Q \ar[ur] \ar[dr] \\
0 \ar[r] &M \ar[r, equals]  &M \ar[rr, "0", swap] \ar[ur]  &&B \ar[r, equals] &B \ar[r] &0
\end{tikzcd}\]    
of its pullback $\xi|_A$ to $A$. A factorization $A \dashrightarrow R$ gives a chain map $\bar 0 \to \xi|_A$:
\[\begin{tikzcd}
 &0  \ar[r] &M  \ar[r, equals] \ar[d]  &M \ar[r, "0"] \ar[d]  &A   \ar[d]\ar[r, equals] &A   \ar[d]\ar[r] &0 \\
\xi|_A : &0 \ar[r]  &M \ar[r]  &N  \ar[r] &R \times_B A  \ar[r] &A \ar[r] &0.
\end{tikzcd}\]
The associated butterfly gives a trivialization of $\xi|_A$. Not all trivializations come from such a factorization. 

\begin{definition}
A \textit{splitting} of a 2-extension $\xi$ is an invertible butterfly $Q : \xi \simeq \bar 0$ to the trivial 2-extension \eqref{eqn:trivial2extn}. Let $\Split(\xi) = \IIsom(\xi, \overline{0})$ be the fibered category of splittings of $\xi$.
\end{definition}

\begin{definition}
An $A$-\emph{algebra structure} on a 2-extension $\xi$ is a splitting of $\xi|_A$. 
\end{definition}

Splittings of $\xi|_A$ can be rearranged into a diagram
\begin{equation}\label{eqn:Aalg2extn}
\begin{tikzcd}
 &0 \ar[r] \ar[d] &N \ar[r] \ar[d, equals] &Q_A \ar[r] \ar[d] &A \ar[r] \ar[d] &0 \\
0 \ar[r] &M \ar[r] &N \ar[r] &R \ar[r] &B \ar[r] &0.
\end{tikzcd}    
\end{equation}
The top row is exact but need not be an extension of algebras $N^2 \neq 0$. Gerstenhaber studied splittings in this form:

\begin{remark}\label{rmk:gerstenhaberdidit}
Gerstenhaber wrote splittings of $\xi$ as
\[\begin{tikzcd}
 &0 \ar[r] \ar[d] &N \ar[r] \ar[d, equals] &Q \ar[r] \ar[d] &B \ar[r] \ar[d, equals] &0 \\
0 \ar[r] &M \ar[r] &N \ar[r] &R \ar[r] &B \ar[r] &0
\end{tikzcd}\]
and referred to them as ``solutions.'' He observed these diagrams form a torsor under $\Exal(B, M)$ \cite[Theorems 2, 5]{gerstenhaber2}. There is an analogous way of writing a splitting as a chain map from $\xi$ \emph{to} a shorter sequence that the reader may define.

The action of an extension 
\[\zeta : \qquad 0 \to M \to B' \to B \to 0\]
on $\xi$ gives:
\[0 \to N \to \bar{R' \times_B B'} \to B \to 0.\]
The middle term $\bar{R' \times_B B'}$ is the pushout of the sum of $\xi$ and $\zeta$ along the addition map $N \times M \to N$. 
\end{remark}

These 2-extensions of $A$-algebras form a 2-category where the butterflies $Q : \xi \simeq \eta$ come with an isomorphism between the two splittings of their pullbacks to $A$. If the $A$-algebra structure of $\xi$ is written \eqref{eqn:Aalg2extn}, the splitting of $\eta$ needs an isomorphism to the composite:
\[\begin{tikzcd}
 &0 \ar[r] \ar[d] &N' \ar[r] \ar[d] &(Q \times_R Q_A)/N \ar[r] \ar[d] &A \ar[r] \ar[d] &0 \\
0 \ar[r] &M \ar[r] &N' \ar[r] &R' \ar[r] &B \ar[r] &0.
\end{tikzcd}\]
The category of $A$-algebra structures on the trivial extension $\bar 0$ is canonically identified with extensions $\Split(\bar 0|_A) = \Exal_A(B, M)$.

Write $\EExal^2_A(B, M)$ for this category of 2-extensions of $A$-algebras. It is the homotopy fiber of the map $\EExal^2(B, M) \to \EExal^2(A, M)$ over the trivial extension. The functor forgetting the $A$-algebra structure $\EExal^2_A(B, M) \to \EExal^2(B, M)$ is an $\EExal(A, M)$-pseudotorsor over each object.

One can define a category $\cal C$ of $A$-algebra structures on 2-extensions na\"ively by asking $A \to B$ to factor through $R$. These are $A$-algebras in our sense together with an additional trivialization $Q \simeq A+M$ of the top row of \eqref{eqn:Aalg2extn}. The transitivity exact sequence \S \ref{s:transles} expressing $\Exal^i$ as the derived functors of $\Der$ does not hold using this notion.

\begin{remark}
The zero object $\bar 0$ of $\EExal_A^2(B, M)$ is the trivial 2-extension $\bar 0$ with the trivial splitting $\bar 0|_A$ as $A$-algebra structure. An isomorphism $\xi \simeq \bar 0$ \emph{as} $A$-algebra extensions is \emph{not} simply a diagram
\[\begin{tikzcd}
 &&0 \ar[r] &N \ar[r] \ar[d] &Q_A \ar[r] \ar[d] &A \ar[r] \ar[d] &0 \\
\zeta : &&0 \ar[r] \ar[d] &N \ar[r] \ar[d] &Q_B \ar[r] \ar[d] &B \ar[r] \ar[d] &0 \\
&0 \ar[r] &M \ar[r] &N \ar[r] &R \ar[r] &B \ar[r] &0.
\end{tikzcd}\]
If the middle row is $\zeta$, this diagram splits $\zeta|_{Q_B}$. A splitting as $A$-algebras need only trivialize an extension of $A$ that restricts to $\zeta|_{Q_B}$.  
\end{remark}

\begin{remark}
Lemma \ref{lem:LSequiv=isom} holds for equivalence classes of $A$-algebra 2-extensions. 
\end{remark}

\section{The topology}

If everything is a cover, nothing is a sheaf and vice versa. J. Wise introduced a topology where deformation problems are locally trivial but they still satisfy descent. 

\begin{definition}[{\cite{wisealgs1}}]\label{defn:wisetop}
A family $\{C_i \to D\}$ in $E$ is covering if, for each finite set of sections $\Lambda \subseteq \Gamma(U, D)$ over an open $U$ of the topos $E$, there exists a single $i$ such that $\Lambda$ lifts to $\Gamma(V, C_i)$ for some epimorphism $V \to U$ in $E$. 

Write $E^*$ for $E$ with this topology instead of the canonical one. 

\end{definition}

\begin{remark}
Epimorphisms in $E$ are coverings, so the topology on $E^*$ is subcanonical. If $E$ arises from a site $\cal C$, a family $\{C_i \to D\}$ is covering if and only if, for any finite set of sections $\Lambda \subseteq \Gamma(U, D)$ over $U \in \cal C$, there is some $C_i$ and a covering sieve $R \subseteq \hh{U}$ such that $\Lambda$ lifts to $\Gamma(R, C_i)$. 

For a family $\{C_i \to D\}$ to be covering, it does not suffice that $\bigsqcup C_i \to D$ is covering. 

If $S \in E$ has associated morphism $j : E/S \to E$ and $\Lambda$ is a finite set, refer to $j_! \underline{\Lambda}$ as a constant sheaf over $S$. Write $I \subseteq E$ for the full subcategory on constant sheaves of finite sets over any object.  A family $\{C_i \to D\}$ is covering if, for all $j_! \underline{\Lambda} \in I$, there is a $C_i$ such that the sheaf of lifts
\[\begin{tikzcd}
        &j^* C_i \ar[d]        \\
\underline{\Lambda} \ar[r] \ar[ur, dashed]         &j^* D 
\end{tikzcd}\]
covers the final object $S$ of $E$. By the Comparison Lemma \cite[Theorem III.4.1]{SGA4}, the categories of sheaves $Sh(I) \simeq Sh(E^*)$ are equivalent. 

\end{remark}

\begin{example}

The topos $E$ with canonical topology is different from $E^*$. If $E = (Sets)$, sheaves on $E^*$ are presheaves on finite sets \cite[Remark pg. 188]{wisealgs1}. These are contravariant, set-valued FI-modules \cite{FImodules}. Cosheaves on $E^*$ are ordinary (covariant) FI-modules, using the identities from \cite{nlab:cosheaf}.

\end{example}

Let $\cal C$ be a site. Recall that $X \in \cal C$ is a \textit{local object} \cite[Definition 0.1]{gabberkellylocal} if any cover $\{Y_i \to X\}$ has a section $X \dashrightarrow Y_i$. Equivalently, there is only one covering sieve of $X$.

\begin{lemma}
Equip the category of sheaves $Sh(\cal C)$ on a site $\cal C$ with the canonical topology. Retracts of the sheaves $\hh{U}$ represented by local objects $U$ of $\cal C$ are local in $Sh(\cal C)$. If every object $F$ of $Sh(\cal C)$ receives an epimorphism from local objects $\bigsqcup \hh{U} \to F$, these are all the local objects of $Sh(\cal C)$

\end{lemma}

\begin{proof}

Covers in $Sh(\cal C)$ are epimorphisms, and every sheaf $F$ has a canonical epimorphism 
\[\bigsqcup_{s \in F(V)} \hh{V} \to F.\]
It suffices to check $F \in Sh(\cal C)$ is local on epimorphisms with source of the form $\bigsqcup \hh{V}$. 

A retraction of a local object $\hh{U}$ is seen to be local. Under the assumption, the canonical epimorphism $\bigsqcup \hh{V} \to F$ can be refined by one with only local objects $V$. If $F$ is local, it admits a section exhibiting $F$ as a retraction of some local $\hh{V}$. 

\end{proof}

\begin{remark}
The local objects in $E^*$ are finite constant sheaves $j_! \underline{\Lambda} \in I$ over \textit{local} objects $S$ of $E$. This is because their standard cover by all the objects of $I$ must be split locally in $S$, but all localizations of $S$ are themselves split. Not all sheaves of finite sets are local. 

\end{remark}

Let $\Aalg/B$ be the category of $A$-algebras over $B$. Endow this category with the topology of Definition \ref{defn:wisetop}. Write $\Gamma_A(B, F), H^p_A(B, F)$ for the global sections and cohomology groups of a sheaf $F$ on $\Aalg/B$. Local objects are split algebra quotients of finite free algebras on local objects of $E$:
\[
\begin{tikzcd}
j_! j^*A[x_1, \dots, x_n] \ar[r]   &C \ar[l, bend right, dashed]
\end{tikzcd}
\]

\begin{example}

One can give the category of $A$-modules $\Amod$ the same topology. Local objects are direct summands of finite free modules over local objects in $E$. If $E = (Sets)$, local objects of $\Amod$ are the finitely generated projective modules. 

\end{example}

\begin{remark}
Algebraic structures such as algebras and non-unital commutative algebras descend in the topology on $E$ as in the proof of \cite[Proposition 2.7]{mymoduledefms}. This topology is also subcanonical on $\Amod, \Aalg$. 

For each finite set $\Lambda \subseteq \Gamma(U, D)$ of sections on some open $j : U \to E$, obtain a map $j_! j^* A[\Lambda] \to D$. As these range over all such finite sets of sections $\Lambda$, we obtain a cover of $D$. 

Every algebra $D$ has this canonical cover by algebras $j_! j^* A[x_1, \dots, x_n]$ which become globally free after localizing. These are different from ``locally free'' modules or algebras on a sheaf of sets $S$, which may still have cohomology. Indeed, cohomology is \emph{defined} via locally free modules $A^S$ in \cite[V.2.1.2]{SGA4}:
\[H^*(S, M) := \Ext^*_A(A^S, M).\]

\end{remark}

Let $\dd{M}$ be the sheaf on $\Aalg/B$ of $A$-derivations with values in $M$:
\[\dd{M}(C) = \Der_A(C, M).\]

Given two 2-extensions $\xi, \eta$, the fibered category $\IIsom(\xi, \eta)$ of invertible butterflies between them is a stack over $\Aalg/B$. This holds for the extensions as modules by \cite[Lemma 4.3]{mymoduledefms}, and we can descend algebraic structures in the topology.

\begin{proposition}\label{prop:isomdMgerbe}
The stack $\IIsom(\xi, \eta)$ of invertible butterflies between any pair of 2-extensions is a gerbe for $\dd{M}$. In particular, $\Split(\xi)$ is a $\dd{M}$-gerbe. 
\end{proposition}

\begin{proof}

The question is local. One can assume $B = A[S]$ is free with a chosen section $S \dashrightarrow R \to B$. Suppose $\xi$ has $A$-algebra structure
\[\begin{tikzcd}
 &&0 \ar[r] \ar[d]  &N \ar[r] \ar[d, equals]  &Q_A \ar[r] \ar[d]  &A \ar[r] \ar[d] &0 \\
\xi : &0 \ar[r]  &M \ar[r]  &N \ar[r]  &R \ar[r]  &A[S] \ar[r]  &0.
\end{tikzcd}\]
The map $Q_A \to R$ has a chosen factorization $Q_A[S] \to R$. The top row consists of $Q_A$-modules and the bottom $R$-modules, so the chain map factors through the tensor product of the top row with $- \otimes_{Q_A} Q_A[S]$:
\[\begin{tikzcd}
0  \ar[r] &N[S] \ar[r] \ar[d]  &Q_A[S] \ar[r] \ar[d] &A[S] \ar[r] \ar[d, equals] &0 \\
0 \ar[r]  &P  \ar[r] &R  \ar[r] &A[S] \ar[r]  &0,
\end{tikzcd}\]
writing $P = \ker (R \to A[S])$. Obtain a chain map of 2-extensions
\[\begin{tikzcd}
0  \ar[r] &M \ar[r] \ar[d, equals]  &N[S] \oplus M \ar[r] \ar[d]  &Q_A[S] \ar[r] \ar[d]  &A[S] \ar[r] \ar[d, equals]  &0 \\
0 \ar[r]  &M \ar[r]  &N \ar[r]  &R \ar[r]  &A[S] \ar[r]  &0.
\end{tikzcd}\]
It suffices to split the top row, which has a canonical splitting 
\[\begin{tikzcd}
 &0 \ar[r] \ar[d]  &N \ar[r] \ar[d]  &Q_A \ar[r] \ar[d]  &A  \ar[r] \ar[d] &0 \\
0 \ar[r]  &M \ar[r] \ar[d, equals]  &N[S] +\epsilon M \ar[r] \ar[d]  &Q_A[S] \ar[r] \ar[d]  &A[S] \ar[r] \ar[d, equals]  &0 \\
0 \ar[r]  &M \ar[r, equals]  &M \ar[r, "0", swap]  &A[S] \ar[r, equals]  &A[S] \ar[r]  &0
\end{tikzcd}\]
\emph{as} $A$-\emph{algebras}. We have a local splitting $\xi \simeq \bar 0$ in $\EExal^2_A(B, M)$. We can assume $\xi = \bar 0$ and $\eta = \bar 0$.

An automorphism $Q$ of $\bar 0$ corresponds to an extension $\zeta \in \EExal_A(B, M)$. Any extension $\zeta$ can be locally split. Splittings of $\zeta$ give isomorphisms of $Q$ with the identity, implying that $Q \simeq id_{\bar 0}$ and $\IIsom(\xi, \eta)$ is locally nonempty and locally connected. 

Automorphisms of the split butterfly $B + \epsilon M : \bar 0 \simeq \bar 0$ coincide with automorphisms of the trivial extension
\[0 \to M \to B + \epsilon M \to B \to 0.\]
These give derivations $B \to M$ \cite[4.3.4]{lichtenbaumschlessinger}. Automorphisms lie over the trivial 2-extension of $A$ by $M$ if and only if they are $A$-derivations. 

\end{proof}

Gerstenhaber obtained this proposition in a rather different form as in Remark \ref{rmk:gerstenhaberdidit}. If $A = \ZZ$, one can localize by $\ZZ[R]$ or $\ZZ[Q]$ to immediately split any 2-extension or butterfly as in the proof.

\section{Deformations of Algebras}

Fix a squarezero extension of rings:
\begin{equation}\label{eqn:fixextn}
\omega : \qquad 0 \to I \to A' \to A \to 0    
\end{equation}
and an $A$-algebra $B$. We equate three explicit sequences comparing $\Der, \Exal, \Exal^2$ of $A$-algebras with $A'$-algebras. 

\subsection{The Transitivity Sequence}\label{s:transles}

There is a straightforward long exact sequence for cohomology $H_A^*(B, \dd{M})$ arising from short exact sequences in the module $M$ \cite[Theorem 3]{gerstenhaber2}. Fix maps of rings $A \to B \to C$ in $E$ and a $C$-module $M$. We describe the ``transitivity triangle'' \cite[\S II.2.1]{illusie1} obtained by varying the algebra $A \to B$:
\begin{equation}\label{eqn:transles}
\begin{tikzcd}
0 \ar[r] &\Der_B(C, M) \ar[r]  &\Der_A(C, M)\arrow[d, phantom, ""{coordinate, name=Z}]
 \ar[r]   &\Der_A(B, M)  \arrow[dll,
"\gamma"
rounded corners,
to path={ -- ([xshift=2ex]\tikztostart.east)
|- (Z) [near end]\tikztonodes
-| ([xshift=-2ex]\tikztotarget.west)
-- (\tikztotarget)}]      \\
&\Exal_B(C, M) \ar[r, "\tau"]  &\Exal_A(C, M) \arrow[d, phantom, ""{coordinate, name=Y}]
 \ar[r] &\Exal_A(B, M) \arrow[dll,
 "\delta",
rounded corners,
to path={ -- ([xshift=2ex]\tikztostart.east)
|- (Y) [near end]\tikztonodes
-| ([xshift=-2ex]\tikztotarget.west)
-- (\tikztotarget)}] \\
&\Exal_B^2(C, M) \ar[r, "\rho"]     &\Exal_A^2(C, M) \ar[r]    &\Exal_A^2(B, M).
\end{tikzcd}    
\end{equation}
Except for the boundary maps $\gamma, \delta$, the maps are induced from evident functorialities of $\Der, \Exal, \Exal^2$ with respect to changing rings. 

The maps $\rho, \tau$ change the algebra structure but not the underlying extension/2-extension. For this to be a complex, $\rho \circ \delta = 0$ and $\tau \circ \gamma = 0$ imply that the underlying extension and 2-extension of $\delta(\zeta), \gamma(\theta)$ must be trivial.

The boundary map $\gamma$ sends a derivation $\theta : B \to M$ to the trivial extension
\[0 \to M \to C + \epsilon M \to C \to 0,\]
but with nontrivial $B$-algebra structure
\[B \to C + \epsilon M\]
given by $\theta$ on the second component.

The map $\delta$ sends an extension 
\[\zeta : \qquad 0 \to M \to B' \to B \to 0\]
to the trivial 2-extension with $\zeta$ as $B$-algebra structure:
\[\begin{tikzcd}
 &0 \ar[r] \ar[d] &M \ar[r] \ar[d, equals] &B' \ar[r] \ar[d] &B \ar[r] \ar[d] &0 \\
0 \ar[r] &M \ar[r, equals] &M \ar[r, "0", swap] &C \ar[r, equals] &C \ar[r] &0.
\end{tikzcd}\]

Sequence \eqref{eqn:transles} is a complex because the $A$-algebra structures on $\gamma(\theta)$ and $\delta(\zeta)$ are always trivial. This is because $\theta$ is an $A$-derivation and $\zeta$ is an $A$-algebra extension.

\begin{lemma}\label{lem:exacttransles}
The sequence \eqref{eqn:transles} is exact. 
\end{lemma}

\begin{proof}

Omitted. 


\end{proof}

\begin{example}

Take $E = (Sets)$ for simplicity. If $C = A[S]$ is a polynomial ring over $A$, the functors $\Exal_A(C, M), \Exal_A^2(C, M)$ vanish to give isomorphisms
\begin{equation}\label{eqn:vanishingdegreeiso}
\Exal_A(B, M) \overset{\delta}{\simeq} \Exal_B^2(C, M)    
\end{equation}
for any intermediate ring $A \to B \to C$. 

A submonoid $P \subseteq \NN^k$ gives such a subalgebra
\[B = A[P] \subseteq C = A[\vec x] = A[x_1, \dots, x_k].\]
Any 2-extension of $A[P]$-algebras
\[0 \to M \to N \to R \to A[\vec x] \to 0\]
must split as a 2-extension of $A$-algebras. Assume therefore that it is $\bar 0$ with trivial $A$-algebra structure. Compatible $A[P]$-algebras on $\bar 0$ constitute extensions $\Exal_A(A[P], M)$. 

\end{example}

\subsection{Deformations}\label{ss:algdefms}

Fix the extension of algebras \eqref{eqn:fixextn}. Given an $A$-algebra $B$ and a $B$-module $M$, we want to classify algebra extensions
\begin{equation}\label{eqn:algdefms}
\begin{tikzcd}
\omega :  &0\ar[r]   &I \ar[r] \ar[d, "\varphi", swap]   &A' \ar[r] \ar[d, dashed]    &A \ar[r] \ar[d]  &0      \\
\zeta : &0\ar[r]   &M \ar[r, dashed]  &B' \ar[r, dashed]    &B \ar[r] &0.
\end{tikzcd}    
\end{equation}
The fibered category of such extensions with a fixed map $\varphi : I \to M$ is a $\dd{M}$-gerbe denoted $\DDef(-, \varphi, M)$ \cite[Proposition in \S 8]{wisealgs1}. 

Diagram \eqref{eqn:algdefms} is equivalent to an isomorphism $\varphi \smile \omega \simeq \zeta|_A$ between the pushforward of the top row and the pullback of the bottom. The stack of deformations is the pullback
\begin{equation}\label{eqn:defmsaspb}
\begin{tikzcd}
\DDef(-, \varphi, M) \ar[r] \ar[d] \pb  &* \ar[d, "\varphi \smile \omega"] \\
\EExal_{A'}(B, M) \ar[r]  &\EExal_{A'}(A, M). 
\end{tikzcd}    
\end{equation}

Algebra extensions $\EExal_{A'}(A, M)$ induce maps $\varphi : I \to M$ between the kernels, and this gives an isomorphism \cite[pg 353]{wisealgs2}
\[\EExal_{A'}(A, M) \simeq \HHom_A(I, M).\]

The extension $\varphi \smile \omega \in \EExal(A, M)$ can be thought of as an $A$-algebra structure on the trivial 2-extension of $B$ by $M$:
\[\begin{tikzcd}
 &0 \ar[r] \ar[d] &I \ar[d, "\varphi"] \ar[r] &A' \ar[r] \ar[d] &A \ar[d] \ar[r] &0 \\
0 \ar[r] &M \ar[r, equals] &M \ar[r, "0", swap] &B \ar[r, equals] &B \ar[r] &0.
\end{tikzcd}\]
Obtain a map 
\[\Hom_A(I, M) \to \Exal^2_A(B, M)\]
by sending $\varphi$ to $\varphi \smile \omega$ written in this form.

\begin{lemma}\label{lem:split=defm}

The splittings $\Split(\varphi \smile \omega)$ give a $\dd{M}$-gerbe equivalent to $\DDef(-, \varphi, M)$. The diagram
\[\begin{tikzcd}
\Hom_A(I, M) \ar[r, "\smile \omega"] \ar[dr, "\DDef", swap] \ar[dr, bend left=15, phantom]        &\Exal_A^2(B, M) \ar[d, "\sim \Split"]        \\
        &H^2_A(B, M)
\end{tikzcd}\]
commutes. 

\end{lemma}

\begin{proof}

Splittings of $\bar 0$ entail extensions $\EExal_{A'}(B, M)$. To make sure such a splitting has correct $A$-algebra structure $\omega$, one restricts to $\EExal_{A'}(A, M)$. This is precisely Diagram \eqref{eqn:defmsaspb}. This identification of gerbes is banded by the identity $\dd{M} \longequals \dd{M}$ because it lies over equivariant inclusions of each gerbe in $\EExal_{A'}(B, M)$. 

\end{proof}

\begin{remark}

One can obtain the same map $\Hom_A(I, M) \to \Exal_A^2(B, M)$ as in \cite[\S 4]{mymoduledefms}. Choose a flat $A'$-algebra with a surjection $P \to B$:
\[0 \to K \to P \to B \to 0,\]
writing $K$ for the kernel. Tensor by $-\otimes_{A'} A$ and write $\bar K = K \otimes_{A'} A$, $\bar P = P \otimes_{A'} A$:
\[\eta : \qquad 0 \to I \otimes_{A'} B \to \bar K \to \bar P \to B \to 0.\]
One gets a map $\Hom_A(I, M) \to \Exal_A^2(B, M)$ independent of choices by sending $\varphi : I \to M$ to the pushout $\varphi \smile \eta$ of $\eta$ along $\varphi$. The diagram 
\[\begin{tikzcd}
 &0 \ar[r] \ar[d] &K \ar[r] \ar[d] &P \ar[r] \ar[d] &B \ar[r] \ar[d, equals] &0 \\
0 \ar[r]  &I \otimes_{A'} B \ar[r]  &\bar K \ar[r]  &\bar P \ar[r]  &B \ar[r]  &0
\end{tikzcd}\]
gives a splitting of $\eta$. The same trivializes the module 2-extension in \cite[\S 4]{mymoduledefms} as groups, but not as $A$-modules. An early draft of this article used a similar approach to defining the map $\delta$ in the transitivity sequence \eqref{eqn:transles}. 

The triangle of Lemma \ref{lem:split=defm} \emph{anti}-commutes in that case as in \cite[Theorem 1.6]{mymoduledefms}. This minus sign reflects a choice in identifying $\Split(\bar 0) \simeq \Exal_{A'}(B, M)$: one can compose splittings in either order. 

\end{remark}

\subsection{Right Derived Functor Sequence}

We get a long exact sequence
\begin{equation}\label{eqn:Tderfunctorles}
0 \to H^1_A(B, M) \to H^1_{A'}(B, M) \to H^1_{A'}(A, M) \to H^2_A(B, M)
\end{equation}
of derived functors. The sequence is exact on the left because the previous term in the sequence vanishes $H^0_{A'}(A, M) = 0$. This also equates two derivations 
\begin{equation}\label{eqn:derequality}
\Der_A(B, M) = \Der_{A'}(B, M).    
\end{equation}
 
We can also apply the transitivity sequence to $A' \to A \to B$.

\begin{lemma}
The long exact sequence \eqref{eqn:transles} for $A' \to A \to B$ coincides with \eqref{eqn:Tderfunctorles} under the isomorphisms of Theorem \ref{thm:algcohomisexal}.
\end{lemma}

\begin{proof}

By Proposition \ref{prop:explicitcohomisexal}, each map sends an object to the torsor/gerbe of local lifts to the previous term. Lemma \ref{lem:split=defm} checks compatibility of the last map. The others check that splittings of an extension $\zeta$ of algebras get sent to the torsor of splittings of the image of $\zeta$. 

\end{proof}

We answer Question \ref{q:howmanydefms?} as per tradition:

\begin{theorem}[{\cite[Theorem 4.3.3]{lichtenbaumschlessinger}, \cite[Proposition 2.1.2.3]{illusie1}}]

Consider the deformation problem \eqref{eqn:algdefms}. There exists a deformation $B'$ if and only if the class $\varphi \smile \omega \in \Exal^2_A(B, M)$ vanishes. If so, deformations $B'$ form a principal homogeneous set under $\Exal_A(B, M)$. Automorphisms of a single deformation correspond to derivations $\Der_A(B, M)$.

\end{theorem}

\subsection{Spectral sequences}

Continue to fix a squarezero algebra extension $\omega$ \eqref{eqn:fixextn} and an $A$-algebra $B$. Write $\Pi$ for the functor
\[\Pi : \Aalg/B \to \AAalg/\Pi B\]
that sends an $A$-algebra $C$ over $B$ to the same ring regarded as an $A'$-algebra or suppress notation $B = \Pi B$. We obtain a morphism of sites by precomposition
\[\pi : \AAalg/\Pi B \to \Aalg/B; \qquad F \mapsto F(\Pi(-))\]
because $\Pi$ is cover-preserving.

Given an $A$-module $M$, derivations as $A'$- and $A$-algebras are the same by \eqref{eqn:derequality}:
\[\Gamma_A(B, \pi_*\dd{M}) = \Gamma_{A'}(B, \dd{M}).\]
The equality of derived functors $R\Gamma_A \circ R\pi_* = R\Gamma_{A'}$ yields a spectral sequence
\[H^p_A(B, R^q \pi_* \dd{M}) \Rightarrow H^{p+q}_{A'}(B, \dd{M}).\]
We get an exact sequence of low degree terms
\begin{equation}\label{eqn:lowdegseq}
0 \to H^1_A(B, \pi_*\dd{M}) \to H^1_{A'}(B, \dd{M}) \to \Gamma_A(B, R^1\pi_* \dd{M}) \to H^2_A(B, \pi_*\dd{M}) \to H^2_{A'}(B, \dd{M}),    
\end{equation}
which can be identified with a sequence of algebra extensions:
\begin{equation}\label{eqn:exallesfromss}
0 \to \Exal_A(B, M) \to \Exal_{A'}(B, M) \to \Gamma_A(B, R^1 \pi_* \dd{M}) \to \Exal^2_A(B, M) \to \Exal^2_{A'}(B, M).    
\end{equation}

The spectral sequence is natural in $B \in \Aalg$, so the exact sequence is as well. Sheafify with respect to $B$ to kill off extensions and 2-extensions of $A$-algebras:
\[(\pi_*\Exal_{A'}(-, M))^{sh} \simeq \Gamma_A(-, R^1\pi_* \dd{M}).\]
The transitivity sequence \eqref{eqn:transles} for $A' \to A \to B$ is similarly natural in $B$, and sheafifying yields an isomorphism:
\[\Exal_{A'}(-, M)^{sh} \simeq \HHom_{A'}(I, M).\]
We have shown the two sequences are equivalent:

\begin{corollary}
The transitivity sequence \eqref{eqn:transles} and that of low degree terms \eqref{eqn:lowdegseq} are equivalent under the isomorphisms of Theorem \ref{thm:algcohomisexal}. 
\end{corollary}

\appendix

\section{Fixing the flaw}\label{a:fixflaw}
\smallskip
\begin{center}Joint with \textsc{Jonathan Wise}\end{center}
\medskip

Theorem \ref{thm:algcohomisexal} computes the cohomology of the sheaf of derivations $\dd{M}$ on the site $\Aalg/B$. There is a small error in the proof explained in Example \ref{freeproductnoiso}. We fix this flaw. 

\begin{theorem}[{\cite[Theorem 4]{wisealgs1}}]\label{thm:algcohomisexal}
Cohomology of the sheaf $\dd{M}$ on $\Aalg/B$ corresponds to the Lichtenbaum-Schlessinger functors $T^0, T^1, T^2, \dots$;
i.e., to $\Ext$ with the cotangent complex: 
\[
\begin{split}
H^p(\Aalg/B, \dd{M}) &= \Ext^p_B(\ccx{B/A}, M)  \\
    &=T^p(B/A, M).
\end{split}
\]
\end{theorem}

The identifications $T^p_A(B, M) \simeq \Exal^p_A(B, M)$ for $p = 1, 2$ appeared in \cite[Theorem 4.1.2, 4.2.2]{lichtenbaumschlessinger}.

\begin{remark}
The higher cohomologies $H^p(\Aalg/B, \dd{M})$ for $p > 2$ are related to Koszul homology. We hope to return to this in future work. 
\end{remark}

\subsection{Limits of free algebras}
\label{sec:limits-of-free-alg}

Let $A[-] : E \rightarrow \Aalg$ be the free algebra functor. We want to show $A[-] : E \to \Aalg$ is something like a ``homotopy equivalence'' of topoi. One cannot simply use the Comparison Lemma \cite[Theorem III.4.1]{SGA4} because the functor is not full. 

The proposition in \cite[\S 9]{wisealgs1} required a morphism of sites $\Psi : \Aalg \dashrightarrow E^*$ whose pullback functor was the free algebra functor $A[-]$.  However, this functor is not left exact, and therefore does not underlie a morphism of sites:

\begin{example}\label{freeproductnoiso}
The map $A[Q \times_S R] \rightarrow A[Q] \times_{A[S]} A[R]$ is not an isomorphism in general. The free algebra functor $A[-]$ therefore does not lead to a morphism of topoi $\Aalg \dashrightarrow E^*$ as claimed in \cite[pg. 189]{wisealgs1}. 

Let $E = (Sets)$, $A = \ZZ$ and $S = \{t\}$. Consider $Q := \{x, y\}$ and $R = \{x', y'\}$ with unique maps to $S$. Then $A[Q \times_S R] \rightarrow A[Q] \times_{A[S]} A[R]$ is not injective. For example, $(x, y) - (x, y') + (x', y') - (x', y)$ goes to zero. Hence the functor $A[-]$ need not commute with finite limits and is not left exact.

\end{example}

Write $A(S)$ for the free $A$-module on a sheaf of sets $S$. The group algebra and free module functors do not preserve products either: 
\[A[G \times H] = A[G] \otimes_A A[H],\]
\[A(Q \times_S R) \neq A(Q) \times_{A(S)} A(R).\]

\begin{example}

	Let $E = (Sets)$, $A = \ZZ$, $R = \{x, y\}$, $S = \{t\}$ and $Q$ be the empty set. The element $(0, x-y) \in A[Q] \times_{A[S]} A[R]$ does not lift to $A[Q \times_S R] = A$.  Thus $A[Q \mathop\times_S R] \to A[Q] \mathop\times_{A[S]} A[R]$ may also fail to be surjective.

\end{example}

The free algebra functor $A[-]$ is ``compatible'' with fiber products of only epimorphisms or only monomorphisms. We prove this compatibility in $E = (Sets)$ and reduce the general case to that. The danger comes from applying $A[-]$ to limits of diagrams with both monomorphisms and epimorphisms.

Write $S^\NN$ for the free commutative monoid on a sheaf $S \in E$ with monoid operation written multiplicatively. The free $A$-algebra functor is the free module on the free commutative monoid: 
\[A[S] = A(S^\NN).\]

If $Q, R \to S$ are epimorphisms in $E$, we show
\[A[Q \times_S R] \to A[Q] \times_{A[S]} A[R]\]
is covering. This will result from the analogous facts for $(-)^\NN, A(-)$.

\begin{lemma} \label{lem:module-case}
	Let $Q, R \to S$ be surjections of sets $E = (Sets)$ and let $A$ be a commutative ring.  Then the induced map of free $A$-modules $A(Q \mathop\times_S R) \to A(Q) \mathop\times_{A(S)} A(R)$ is surjective.
\end{lemma}
\begin{proof}
Suppose $x$ lies in $A(Q) \mathop\times_{A(S)} A(R)$.  We would like to lift $x$ to $A(Q\mathop\times_S R)$.  Since $Q \mathop\times_S R \to R$ is surjective, so is $A(Q\mathop\times_S R) \to A(R)$, so we may find an element of $A(Q\mathop\times_S R)$ with the same image as $x$ in $A(R)$.  Subtracting the image of this element from $x$, we may assume that the image of $x$ in $A(R)$ is $0$.

Let $f : Q \to S$ denote the projection.  The kernel of $A(Q) \mathop\times_{A(S)} A(R) \to A(R)$ coincides with the kernel of $A(Q) \to A(S)$.  We can therefore write $x$ as $\sum a_q q$ with $a_q \in A$ and $\sum_{q \mapsto s} a_q = 0$ for each $s \in S$.  For each of the finitely many $s \in S$ that appear as images of $q$ with $a_q \neq 0$, choose $r(s) \in R$ whose image under $R \to S$ is $s$.  Then $\sum_{q \in Q} a_q (q, r(f(q)))$ lies in $A(Q \mathop\times_S R)$ and projects to $\sum_{q \in Q} a_q q = x$ in $A(Q)$, as required.
\end{proof}

\begin{lemma} \label{lem:monoid-case}
	Let $Q, R \to S$ be surjections of sets $E = (Sets)$.  The induced map of free commutative monoids 
	\[(Q \mathop\times_S R)^{\mathbf N} \to Q^{\mathbf N} \mathop\times_{S^{\mathbf N}} R^{\mathbf N}\]
	is surjective.
\end{lemma}
\begin{proof}
Let $\prod_{i = 1}^n q_i$ and $\prod_{j=1}^m r_j$ have the same image in $S^{\mathbf N}$.  Then we must have $n = m$ and, after reordering, we can assume that $q_i$ and $r_i$ have the same image in $S$.  Then $\prod (q_i, r_i)$ is an element of $(Q \mathop\times_S R)^{\mathbf N}$ that projects to $\prod q_i \in Q^{\mathbf N}$ and to $\prod r_i \in R^{\mathbf N}$.
\end{proof}

\begin{lemma} \label{lem:ring-case}
	Let $Q, R \to S$ be surjections of sets $E = (Sets)$ and $A$ a commutative ring.  The ring homomorphism
	\begin{equation*}
		A[Q \mathop\times_S R] \to A[Q] \mathop\times_{A[S]} A[R]
	\end{equation*}
	is surjective.
\end{lemma}
\begin{proof} 
	By Lemmas~\ref{lem:module-case} and~\ref{lem:monoid-case}, the composition
	\begin{equation*}
		A[Q \mathop\times_S R] = A((Q \mathop\times_S R)^{\mathbf N}) \to A(Q^{\mathbf N} \mathop\times_{S^{\mathbf N}} R^{\mathbf N}) \to A(Q^{\mathbf N}) \mathop\times_{A (S^{\mathbf N})} A(R^{\mathbf N}) = A[Q] \mathop\times_{A[S]} A[R]
	\end{equation*}
	is surjective.
\end{proof}

Return to an arbitrary topos $E$.

\begin{proposition}\label{coveringfreeproduct}
	Take epimorphisms $Q, R \to S$ in a topos $E$.  The map 
	\begin{equation} \label{eqn:1}
A[Q \times_S R] \to A[Q] \times_{A[S]} A[R]
	\end{equation}
is covering. 

\end{proposition}

\begin{proof}
	We must show that~\eqref{eqn:1} is surjective.  If $x$ is any section of $A[Q] \mathop\times_{A[S]} A[R]$ over an object $U$ of $E$ then there is a cover $U'$ of $E$ and finite sets $Q'$, $R'$, and $S'$ fitting into a commutative diagram~\eqref{eqn:2} such that $x$ is induced from $A[Q'] \mathop\times_{A[S']} A[R']$:
	\begin{equation} \begin{tikzcd}
		\label{eqn:2}
		Q'_{U'} \ar[r] \ar[d] & S'_{U'} \ar[d] & R'_{U'} \ar[l] \ar[d] \\
		Q \big|_{U'} \ar[r] & S \big|_{U'} & R \big|_{U'} \ar[l]
	\end{tikzcd}
	\end{equation}
	The subscript ${-}_{U'}$ above denotes a constant sheaf over $U'$ and $-\big|_{U'}$ denotes restriction to $U'$.  Replacing $E$ with the slice $E/U'$ and replacing $Q$, $R$, and $S$ with $Q'$, $R'$, and $S'$, the lemma reduces to the case where $Q$, $R$, and $S$ are all constant sheaves, which is Lemma~\ref{lem:ring-case}.
\end{proof}

\begin{example}\label{ex:freealgequalizers}

The functor $A[-]$ does not take all finite limits of epimorphisms to covers, just fiber products: letting $Q = \lim (R \rightrightarrows S)$ be the equalizer, the induced map $A[Q] \to \lim (A[R] \rightrightarrows A[S])$ may fail to be covering. For example, take $E = (\mathit{Sets})$, $A = \ZZ$, $R = S = \{x, y\}$ and the maps $R \rightrightarrows S$ given by the identity and the map swapping $x$ and $y$. The equalizer of sets $Q = \varnothing$ is the empty set and $\ZZ[Q] = \ZZ$. The equalizer of algebras consists of symmetric functions such as $x + y$ in the variables $x, y$. 

\end{example}

\begin{remark}

The functor taking sheaves to free algebras $S \mapsto A[S]$ preserves fiber products of \emph{injections} of sheaves. It does not preserve all limits of injections of sheaves as in Example \ref{ex:freealgequalizers}. 

Presenting the equalizer $Q$ of the example as $R \times_{S \times S} S$, we get $A[Q] = A[R] \times_{A[S \times S]} A[S]$. This differs from $A[R] \times_{A[S] \times A[S]} A[S]$, the equalizer of the algebras. 

\end{remark}

We get a morphism of sites from $A[-]$ in the other direction. 

\begin{proposition}
The free algebra functor $F : E^\ast \rightarrow \Aalg$ is cocontinuous. 
\end{proposition}
\begin{proof}
	Suppose that $T$ is an object of $E$ and let $\{ B_i \to A[T] \}$ be a covering family in $\Aalg$.  If $U$ is an object of $E$ and $\Lambda \subset \Gamma(U, T)$ is a finite set of sections of $T$ over $U$ then we can also view $\Lambda$ as a finite set of sections of $A[T]$ over $U$.  By definition of the topology of $\Aalg$, the $U'$ over $U$ for which $\Lambda \big|_{U'}$ lifts to some $B_i$ form a cover of $U$.
\end{proof}

Write $f : E \to \Aalg$ for the morphism of sites induced from the cover-lifting $A[-]$. The pullback $f^*$ is
\[f^* B(U) = B(A[U]):\]
We don't have to sheafify because $A[-]$ is also cover preserving \cite[00XR]{stacks-project}.

\subsection{\v Cech Cohomology}

This section proves \cite[Corollary, pg. 190]{wisealgs1}. The rest of the proof of Theorem \ref{thm:algcohomisexal} is unchanged. 

Fix $S \in E$ and let $J$ be an injective $A[S]$-module throughout this section. Write $\dd{J}$ the sheaf of $A$-derivations valued in $J$. For an object $X$, we also write $X$ for the sheaf $\hh{X}$ it represents.  We continue to write $f : E \to \Aalg$ for the morphism of sites induced by the cocontinuous functor sending $S$ to $A[S]$, as described Section~\ref{sec:limits-of-free-alg}.

\begin{theorem}[{\cite[Corollary, pg. 190]{wisealgs1}}]\label{thm:injdervanishing}

The higher cohomology of $\dd{J}$ on $\Aalg/A[S]$ vanishes:
\[H^p(\Aalg/A[S], \dd{J}) = 0 \qquad p \neq 0.\]

\end{theorem}

\begin{lemma}
The pullback $f^*\dd{J}$ of derivations to $E$ is represented by $J$:
\[f^* \dd{J} = J.\]
\end{lemma}

\begin{proof}
Consider its value on some $T \in E_{/S}$: 
\[\begin{split}
    (f^* \dd{J})(T)      &:= \dd{J}(A[T])       \\
            &:= \Der_A(A[T], J)     \\
	    &\simeq \Hom_{\Aalg/A[S]}(A[T], A[S] + \epsilon J) \\
            &\simeq \Hom_{E}(T, J)
\end{split}\]
	The $A$-algebra $A + \epsilon J$ is the trivial square-zero extension of $A$ by $J$.  Its map to $A[S]$ sends $\epsilon J$ to $0$.
\end{proof}

\begin{proof}[Proof of Theorem \ref{thm:injdervanishing}]

It suffices to show \v Cech cohomology vanishes.

Let $\scr U = \{T_i \to S\}$ be a cover in $E$ and write $\scr U_\bullet$ for the \v Cech hypercover
\[\scr U_p = \bigsqcup T_{i_0} \times_S \cdots \times_S T_{i_p}.\]
Let $\scr V = \{A[T_i] \to A[S]\}$ be the associated cover of free algebras and $\scr V_\bullet$ the \v Cech hypercover. 

We consider the map $A[\scr U_\bullet] \to \scr V_\bullet$ given on $p$-simplices by the disjoint union of the maps
\begin{equation}\label{eqn:mapsbetweencechs}
A[T_{i_0} \times_S \cdots \times_S T_{i_p}] \longrightarrow A[T_{i_0}] \times_{A[S]} \cdots \times_{A[S]} A[T_{i_p}]    
\end{equation}
	of representable sheaves on $\Aalg$.  These maps are all covering by inductive application of Proposition~\ref{coveringfreeproduct}.

We show the \v Cech cohomology of $\dd{J}$ vanishes on $\scr V_\bullet$ by comparing with $\dd{J}$ on $A[\scr U_\bullet]$:
\begin{equation}\label{eqn:cechcomparison}
\begin{tikzcd}
\cdots &\dd{J}(A[\scr U_{p+1}]) \ar[l]     &\dd{J}(A[\scr U_{p}]) \ar[l, "\partial^p", swap]   &\dd{J}(A[\scr U_{p-1}]) \ar[l, "\partial^{p-1}", swap]    &\cdots \ar[l]       \\
\cdots &\dd{J}(\scr V_{p+1}) \ar[u, hook] \ar[l]       &\dd{J}(\scr V_{p}) \ar[u, hook, "t"] \ar[l, "\partial^p"]       &\dd{J}(\scr V_{p-1}) \ar[u, hook] \ar[l, "\partial^{p-1}"]       &\cdots \ar[l]
\end{tikzcd}    
\end{equation}
The chain maps $\partial^p$ come from alternating sums of face maps for the simplicial hypercovers. The vertical maps are injective because the maps \eqref{eqn:mapsbetweencechs} are covers. The top row is exact because it computes the \v Cech cohomology of the injective sheaf $f^* \dd{J} = J$ on $\scr U_\bullet$. 

Given a section $\alpha \in \dd{J}(\scr V_p)$ that maps to zero under $\partial^p$, the section $t \alpha$ is in the image of $\partial^{p-1}$ by exactness of the top row. If $\partial^{p-1} \beta = t \alpha$, we argue that $\beta$ lies in $\dd{J}(\scr V_{p-1})$. Then $\beta$ maps to $\alpha$.

\begin{proposition}\label{rmk:toposepipushout}

Let $X, Y \to Z$ be epimorphisms in a topos. The pullback square
\[\begin{tikzcd}
X \times_Z Y \ar[r] \ar[d] \pb \ar[dr, phantom, very near end, "\lrcorner"]       &X \ar[d]      \\
Y \ar[r]       &Z
\end{tikzcd}\]
is also a pushout.
\end{proposition}
\begin{proof}
	Suppose that $f : X \to F$ and $g : Y \to F$ are morphisms that agree on $X \mathop\times_Z Y$.  Since $Y \to Z$ is covering, the map $X \mathop\times_Z Y \to X \mathop\times_Z X$ is covering, and the two maps $fq, fp : X \mathop\times_Z X \to F$ agree.  Epimorphisms are effective in a topos, so this implies that $f$ descends to a morphism $h : Z \to F$.  A similar argument shows that $g$ descends to $h' : Z \to F$.  But $h$ and $h'$ agree on the cover $X \mathop\times_Z Y$ of $Z$, so we must have $h = h'$.
\end{proof}

\begin{lemma}\label{lem:cechpushout}
For each face map $d_j$, the following diagram of sheaves on $\Aalg$ is cocartesian:
\begin{equation} \label{eqn:3}
\begin{tikzcd}
A[\mathscr U_p] \ar[d] \ar[r, "d_j"] \ar[dr, phantom, very near end, "\lrcorner"]         &A[\scr U_{p-1}] \ar[d]        \\
\mathscr V_p \ar[r, "d_j", swap]        &\scr V_{p-1}
\end{tikzcd}
\end{equation}
\end{lemma}

\begin{proof}
	For any subset $I = \{ i_0, \ldots, i_p \}$ of the indexing set of the cover $\mathscr U = \{ T_i \}$, let us write $\mathscr U_I$ for $T_{i_0} \mathop\times_S \cdots \mathop\times_S T_{i_p}$ and $\mathscr V_I$ for $A[T_{i_0}] \mathop\times_{A[S]} \cdots \mathop\times_{A[S]} A[T_{i_p}]$.  Choose one index $j$ and let $I' = I \smallsetminus \{ i_j \}$.  Then Diagram~\eqref{eqn:3} is the disjoint union of Diagrams~\eqref{eqn:5} over all $I$: 
	\begin{equation} \label{eqn:5}
	\begin{tikzcd}
		A[\mathscr U_I] \ar[r] \ar[d] & A[\mathscr U_{I'}] \ar[d] \\
		\mathscr V_I \ar[r] & \mathscr V_{I'}
	\end{tikzcd}
	\end{equation}
	We wish to show that each diagram~\eqref{eqn:5} is cocartesian.

	By Proposition~\ref{coveringfreeproduct}, the following natural morphism is covering:
	\begin{equation*}
		A[\mathscr U_I] = A[T_{i_j} \mathop\times_S \mathscr U_{I'}] \to A[T_{i_j}] \mathop\times_{A[S]} A[\mathscr U_{I'}] = \mathscr V_I \mathop\times_{\mathscr V_{I'}} A[\mathscr U_{I'}]
	\end{equation*}
	Combined with Lemma~\ref{rmk:toposepipushout}, this observation shows that~\eqref{eqn:5} is cocartesian.
\end{proof}

Since $\dd{J}$ is a sheaf, Lemma~\ref{lem:cechpushout} implies that the squares in~\eqref{eqn:cechcomparison} are cartesian.  The diagram of associated normalized chain complexes is a direct summand of diagram~\eqref{eqn:cechcomparison}~\cite[\href{https://stacks.math.columbia.edu/tag/019I}{Tag 019I}, (3)]{stacks-project}.  Its squares are therefore cartesian as well.  Furthermore, the normalized chain complexes compute \v Cech cohomology~\cite[\href{https://stacks.math.columbia.edu/tag/019I}{Tag 019I}, (5)]{stacks-project}.

We may now complete the proof of Theorem~\ref{thm:injdervanishing}.  Suppose that $\alpha$ is a normalized cocyle in $\dd J(\mathscr V_p)$ representing a class in $H^p(\Aalg, \dd J)$.  Then $t(\alpha)$ is a normalized cocyle in $\dd J(A[\mathscr U_p]) = \Gamma(\mathscr U_p, J)$ representing a class in $H^p(S, J)$.  But $J$ is injective, so $t(\alpha)$ lifts to $\Gamma(\mathscr U_{p-1}, J) = \dd J(A[\mathscr U_{p-1}])$.  Since the normalized squares of~\eqref{eqn:cechcomparison} are cartesian, it follows that there is some $\beta \in \dd J(\mathscr V_{p-1})$ lifting $\alpha$, which implies that $\alpha$ represents $0$ in $H^p(\Aalg, \dd J)$.

\end{proof}

\begin{remark}

The proof of Theorem \ref{thm:injdervanishing} can be mimicked to provide another proof of \cite[Theorem 1.5]{mymoduledefms} for free modules $M = A(S)$. The case of a general $A$-module $M$ can be computed by \v Cech cohomology on a free simplicial resolution $P_\bullet \to M$ and reduced to this case. 

\end{remark}

\begin{remark}

More generally, suppose $F : \cal C \to \cal D$ is a cover-preserving and cover-lifting functor. Write $(f^*, f_*)$ for the morphism of sites $Sh(\cal C) \to Sh(\cal D)$ with
\[(f^* P)(S) = P(FS).\]

Suppose $F$ sends fiber products of epimorphisms to covers
\[F(Q_1 \times_S \cdots \times_S Q_p) \longrightarrow F(Q_1) \times_{F(S)} \cdots \times_{F(S)} F(Q_p)\]
as in Lemma \ref{coveringfreeproduct}. If $M$ is a sheaf of abelian groups on $\cal D$ such that $f^* M$ is injective and additive, then the higher cohomology vanishes:
\[H^p(Sh(\cal D)/f_! S, f^* M) = 0 \qquad p \neq 0.\]
The proof is the same.

\end{remark}

\bibliographystyle{alpha}
\bibliography{zbib}

\end{document}